\documentclass{amsart} 
\usepackage{amssymb} 
\usepackage{newtxtext}
\usepackage{xpatch}
\usepackage{txfonts}
\usepackage[all]{xy}
\usepackage{amscd}
\usepackage{amsmath}
\usepackage{mathrsfs}
\usepackage{mathtools}
\usepackage[dvipdfmx]{graphicx}
\mathtoolsset{showonlyrefs=true}

\theoremstyle{normal} 
\newtheorem{Thm}{\indent\sc Theorem}[section]
\newtheorem{Lemma}[Thm]{\indent\sc Lemma}
\newtheorem{Coro}[Thm]{\indent\sc Corollary}
\newtheorem{Prop}[Thm]{\indent\sc Proposition}

\theoremstyle{definition} 
\newtheorem{Def}[Thm]{\indent\sc Definition}
\newtheorem{Rem}[Thm]{\indent\sc Remark}
\newtheorem{Ex}[Thm]{\indent\sc Example}
\newtheorem{Nota}[Thm]{\indent\sc Notation}

\newtheorem{Const}[Thm]{Construction}
\newcommand{\Spec}{\mathrm{Spec}\ }
\newcommand{\CC}{\mathbb{C}}

\begin{document}
\title{On the configuration of the singular fibers of jet schemes of rational double points}

\author[Y. Koreeda]{Yoshimune Koreeda}

\subjclass[2010]{ 
Primary 14J17
}
%
\keywords{Jet scheme; rational double point singularities}
\address{
Department of Mathematics, Graduate School of Science \endgraf
Hiroshima University \endgraf
1-3-1 Kagamiyama, Higashi-Hiroshima, 739-8526 \endgraf
Japan
}
\email{d193613@hiroshima-u.ac.jp}


	\maketitle

	\begin{abstract}
To each variety $X$ and a nonnegative integer $m$,
there is a space $X_m$ over $X$,
called the jet scheme of $X$ of order $m$,
parametrizing $m$-th jets on $X$.
Its fiber over a singular point of $X$ is called a singular fiber.
For a surface with a rational double point,
Mourtada(\cite{M2}) gave a one-to-one correspondence between the irreducible components of the singular fiber of $X_m$ and the exceptional curves of the minimal resolution of $X$ for $m \gg 0$.

In this paper,
for a surface $X$ over $\CC$ with a singularity of $A_n$ or $D_4$-type,
we study the intersections of irreducible components of the singular fiber and construct a graph using this information.
The vertices of the graph correspond to irreducible components of the singular fiber and two vertices are connected when the intersection of the corresponding components is maximal for the inclusion relation.
In the case of $A_n$ or $D_4$-type singularity,
we show that this graph is isomorphic to the resolution graph for $m \gg 0$.
	\end{abstract}
	\section*{Introduction}

Let $X$ be a surface over $\CC$ and $\mathrm{Sing}\ X$ the singular locus of $X$.
The notions of a jet scheme and an arc scheme were introduced by J. F. Nash in 1968 in a preprint,
later published in 1995(\cite{Na}).
Roughly speaking,
an $m$-th jet of $X$ is an infinitesimal map of order $m$ from a germ of a curve to $X$,
and an arc of $X$ is an infinitesimal map of order infinity from a germ of a curve to $X$.
The $m$-th order jet scheme of $X$,
denoted by $X_m$,
is a scheme parametrizing $m$-th jets,
and the arc scheme of $X$,
denoted by $X_{\infty}$,
is a scheme parametrizing arcs.
For nonnegative integers $m > m'$,
there is a map $\pi_{m,m'} : X_m \rightarrow X_{m'}$,
called the truncation morphism.
Then an arc scheme can be obtained as the projective limit of jet schemes with respect to the truncation morphisms.

The $0$-th jet scheme $X_0$ is identified with $X$,
and hence we have the morphism $\pi_{m,0}:X_m \rightarrow X$.
We call $X_m^0 = \pi_{m,0}^{-1}(\mathrm{Sing}\ X)$ the singular fiber.
It is hoped that jet schemes and arc schemes,
and in particular the singular fibers,
reflect the property of singular points.
For arc schemes,
there is a famous problem called the Nash problem.
The Nash problem asks the relation between the ``Nash components" and the ``essential divisors" of a resolution of singularities.
There is a natural injective map,
called the Nash map,
from the set of Nash components to the set of essential divisors.
The problem is whether the Nash map is surjective.
This problems was affirmatively solved for $A_n$-type singular surfaces in \cite{Na},
for the remaining rational double points in \cite{{Pe}, {Pl}, {PS}},
for rational surface singularities in \cite{{Re1}, {Re2}},
and for arbitrary surface singularities in \cite{BP}.

On the other hand,
the study of relations between the singular fibers of \emph{jet} schemes and the exceptional divisors of a resolution of singularities for surfaces was only recently started,
in a series of papers by Mourtada(\cite{M1}, \cite{M2}) and Mourtada-Pl\'{e}nat(\cite{MP}).
For a general surface $X$,
the relation between the irreducible components of the singular fiber $X_m^0$ and the exceptional divisors of the minimal resolution of singularities is not simple.
For example,
the number of irreducible components of $X_m^0$ and the number of exceptional divisors are not necessarily equal even for $m \gg 0$.
However,
for rational double point singularities,
Mourtada(\cite{M1}, \cite{M2}) gave a one-to-one correspondence between the irreducible components of the singular fiber of $X_m$ for a fixed $m \gg 0$ and the exceptional curves of the minimal resolution of $X$.
Moreover,
in \cite{MP},
Mourtada and Pl\'{e}nat define ``essential components" and ``minimal embedded toric resolutions".
An essential component is an irreducible component of $X_m^0$ satisfying certain conditions,
where $m$ is allowed to vary.
They gave a one-to-one correspondence between the set of the essential components and the set of the divisors which appear on every ``minimal embedded toric resolution" for rational double point singularities except for the $E_8$-type singular surface.
Further,
they show how to obtain the minimal embedded toric resolution from the information of essential components.

In this paper,
we consider the following quastion:
What one can get from the correspondence between irreducible components of $X_m^0$ for a \emph{fixed} $m \gg 0$ and exceptional curves of the minimal resolution of singularity?
In the case $X$ is $A_n$ or $D_4$-type singular surface over $\CC$,
we study how the irreducible components of $X_m^0$ intersect with each other for $m \gg 0$,
and construct a graph using this information.
This graph will be isomorphic to the resolution graph.

Let us explain how to construct the graph.
We expect that if two irreducible components correspond to distant vertices on the resolution graph,
then their intersection is ``small" in some sense.
A naive expectation would be that the intersection has lower dimensions,
but this is not true.
We explicitly calculate the intersections for an $A_n$-type singular surface and see that the dimensions of the intersections of two distinct irreducible components are independent of the choice of irreducible components for $m \gg 0$.
Still,
we can determine the adjacency as follows.
Let $Z_m^1,...,Z_m^n$ be the irreducible components of the singular fiber $X_m^0$.
	\begin{Const}\label{ConstructiontheGraph}
Let $V = \{ Z_m^1,...,Z_m^n \}$,
and let $E \subseteq \{ Z_m^i \cap Z_m^j \ |\  i,j \in \{1,...,n\}\ \mathrm{with}\  i \neq j \}$ be the set of the maximal elements for the inclusion relation.
Then we construct a graph $\Gamma$ as the pair $(V, E)$,
i.e.
the vertices of $\Gamma$ are elements of $V$,
and there is given an edge between $Z_m^i$ and $Z_m^j$ if and only if $Z_m^i \cap Z_m^j \in E$.
 	\end{Const}
To study the intersections,
we use the description of the irreducible components of the singular fiber by Mourtada.
For $A_n$-type singular surfaces,
Mourtada(\cite{M1}) gave generators of the defining ideals of the irreducible components of $X_m^0$.
Hence it is possible to obtain the irreducible decompositions of $Z_m^i \cap Z_m^j$ with $i \neq j$,
and this enables us to describe the graph $\Gamma$ in Construction \ref{ConstructiontheGraph} for $A_n$-type singular surfaces.
This graph is isomorphic to the resolution graph of an $A_n$-type singular surface.
For a $D_4$-type singular surface,
Mourtada(\cite{M2}) describes the irreducible components of the singular fiber as the closures of certain locally closed sets.
Thus we do not know the generators of the defining ideals of the irreducible components of $X_m^0$.
Still,
we can find certain elements of the defining ideals of $Z_m^i \cap Z_m^j$,
which allow us to study the inclusion relations.
In this way we determine the graph $\Gamma$ in Construction \ref{ConstructiontheGraph} for a $D_4$-type singular surface.
This graph is also isomorphic to the resolution graph of a $D_4$-type singular surface.
We expect that,
for rational double point singularities,
the graphs defined in the same way are isomorphic to the resolution graphs.

The organization of this paper is as follows.
In section 2,
we fix some notations on jet schemes.
In section 3,
we recall the description of the defining ideals of irreducible components of the singular fiber of an $A_n$-type singular surface by Mourtada(\cite{M1}).
Then we study the intersections of irreducible components,
in particular the irreducible decompositions and the dimensions of the intersections.
In section 4,
using the description of irreducible components of the singular fiber of a $D_4$-type singular surface by Mourtada(\cite{M2}),
we determine the maximal elements of their intersections.

\textit{Acknowledgement.}
The author would like to thank Nobuyoshi Takahashi for valuable advice.

	\section{Jet schemes}

In this section, we recall the definition of jet schemes,
describe them explicitly and fix some notations.
We are interested in a neighborhood of a singularity,
so we consider an affine scheme of finite type over an algebraically closed field $k$ as a target space.

Let $X$ be an affine scheme of finite type over $k$ and let $m$ be a nonnegative integer.

	\begin{Prop} (\cite[Proposition 2.2]{Is})
Let ${\mathbf{Sch}}/k$ denote the category of schemes over $\Spec k$ and ${\mathbf{Set}}$ the category of sets.
We define the functor
		\begin{center}
	$F_{m}^{X} : {\mathbf{Sch}}/k \rightarrow {\mathbf{Set}}$
		\end{center}
as follows:
For $Z \in {\mathbf{Sch}}/k$,
		\begin{center}
	$F_m^X (Z) := {\rm Hom}_k(Z \times_{\Spec k} \Spec k[t]/\langle t^{m+1}\rangle,X)$.
		\end{center}
Then $F_m^X$ is represented by an affine scheme $X_m$ of finite type over $k$.
This scheme $X_m$ is called the $m$-th jet scheme of $X$.
	\end{Prop}

For an affine scheme of finite type over $k$,
we can describe its $m$-th jet scheme as follows.

Let $X$ be an affine scheme embedded in $\mathbb{A}^e$.
Then its affine coordinate ring $\Gamma(X, \mathscr{O}_X)$ can be written in the form $k[x_1,...,x_e]/\langle f_1,...,f_r\rangle$.
We introduce some notations.
	\begin{Nota}
Let $\mathbf{x}_i := x_i^{(0)} + x_i^{(1)}t + \cdots + x_i^{(m)}t^m \in k[x_1^{(0)},...,x_1^{(m)},...,x_e^{(0)},...,x_e^{(m)},t]/ \langle t^{m+1} \rangle$ $(i = 1,...,e)$.
For a polynomial $f \in k[x_1,...,x_e]$,
we expand $f(\mathbf{x}_1,...,\mathbf{x}_e)$ as
	\begin{center}
$f \left( \mathbf{x}_1,...,\mathbf{x}_e \right) = f^{(0)} + f^{(1)}t + \cdots + f^{(m)}t^m$ 
	\end{center}
in $k[x_1^{(0)},...,x_1^{(m)},...,x_e^{(0)},...,x_e^{(m)},t]/ \langle t^{m+1} \rangle$,
where $f^{(j)} \in k[x_1^{(0)},...,x_1^{(m)},...,x_e^{(0)},...,x_e^{(m)}]$.
Then the $m$-th jet scheme $X_m$,
which represents the functor $F_m^X$,
is
	\begin{center}
$X_m = \Spec (k[x_1^{(0)},...,x_1^{(m)},...,x_e^{(0)},...,x_e^{(m)}]/\langle f_1^{(0)},...,f_1^{(m)},...,f_r^{(0)},...,f_r^{(m)}\rangle)$.
	\end{center}
	\end{Nota}
	\begin{Rem} \label{jet polynomial of order m can be calculate as more than the order m variables}
Let $g \in k[x_1,...,x_e]$ and $m \in \mathbb{Z}_{\geq 0}$.
The polynomials $g^{(j)}$ are independent of $m$ as long as $m \geq j$.
In particular,
if we want to calculate the polynomial $g^{(j)}$,
we have only to calculate the polynomial $\displaystyle g \left( \sum_{k = 0}^{j}x_1^{(k)}t^k, ..., \sum_{k=0}^{j}x_e^{(k)}t^k \right)$.
	\end{Rem}
	\begin{Ex}
(1)
Suppose $X = \mathbb{A}^e$.
Then the $m$-th jet scheme $X_m$ is $\mathbb{A}^{e(m+1)}$.
\\
(2)
We calculate $X_2$ for $X = \Spec (\CC[x,y,z]/\langle xy - z^2 \rangle)$.
Let $f = xy - z^2$,
$\mathbf{x} = x_{0} + x_{1}t + x_{2}t^2$,
$\mathbf{y} = y_{0} + y_{1}t + y_{2}t^2$ and $\mathbf{z} = z_{0} + z_{1}t + z_{2}t^2$.
Then
	\begin{align}
f(\mathbf{x},\mathbf{y},\mathbf{z}) =& (x_{0}y_{0} - z_{0}^{2}) + (x_{1}y_{0} + x_{0}y_{1} - 2z_{0}z_{1})t\\
 &  + (x_{2}y_{0} + x_{1}y_{1} + x_{0}y_{2} - z_{1}^{2} - 2z_{0}z_{2})t^2 + Ft^3
	\end{align}
where $F \in \CC[x_{0},x_{1},x_{2},y_{0},y_{1},y_{2},z_{0},z_{1},z_{2},t]$.
We set
	\begin{align}
f^{(0)} =& x_{0}y_{0} - z_{0}^{2},\\[-5pt]
f^{(1)} =& x_{1}y_{0} + x_{0}y_{1} - 2z_{0}z_{1},\\[-5pt]
f^{(2)} =& x_{2}y_{0} + x_{1}y_{1} + x_{0}y_{2} - z_{1}^{2} - 2z_{0}z_{2}.
	\end{align}
The second jet scheme of $X$ is
	\begin{center}
$X_2 = \Spec (\CC[x_{0},x_{1},x_{2},y_{0},y_{1},y_{2},z_{0},z_{1},z_{2}]/\langle f^{(0)},f^{(1)},f^{(2)} \rangle)$.
	\end{center}
	\end{Ex}

Next we consider closed points of $X_m$.
In the above situation,
the scheme $X_m$ is a closed subvariety of $(\mathbb{A}^e)_m = \mathbb{A}^{e(m+1)}$,
so we regard the closed points of $X_m$ as an element of $k^{e(m+1)}$.
	\begin{Nota}
Let $\gamma = (a_1^{(0)},...,a_1^{(m)},...,a_e^{(0)},...,a_e^{(m)}) \in \mathbb{A}^{e(m+1)}$ be a closed point.
Then we also denote $\displaystyle \gamma = \left(\sum_{i=0}^{m}a_{1}^{(i)}t^i,...,\sum_{i=0}^m a_{e}^{(i)}t^i \right)$ using the variable $t$.

For $g \in k[x_1,...,x_e]$,
we regard $g$ as a morphism $\mathbb{A}^e \rightarrow \mathbb{A}^1$,
and then the composition $g \circ \gamma$ is given by the substitution as $\displaystyle g\left(\sum_{i=0}^{m}a_{1}^{(i)}t^i,...,\sum_{i=0}^m a_{e}^{(i)}t^i \right)$.
Then we define $\mathrm{ord}_{\gamma}(g)$ as the $t$-order of $g \circ \gamma$ in $k[t]$.
	\end{Nota}
	\begin{Rem}
A closed point $\alpha \in (\mathbb{A}^e)_m = \mathbb{A}^{e(m+1)}$ belongs to $X_m$ if and only if $\mathrm{ord}_{\alpha}(f_i) \geq m+1$
for $i = 1,...,r$.
	\end{Rem}

Finally we see how $X_m$ and $X_{m'}$ are related for $m,m' \in \mathbb{Z}_{\geq 0}$ with $m \geq m'$.
Let $Z$ be any scheme over $\Spec k$.
The ring homomorphism
	\begin{center}
		$k[t]/\langle t^{m+1}\rangle \rightarrow k[t]/\langle t^{m'+1}\rangle\ ;\ \sum_{i = 0}^m \alpha_i t^i \mapsto \sum_{i=0}^{m'} \alpha_i t^i$
	\end{center}
induces the morphism of affine schemes
	\begin{center}
		$\Spec k[t]/\langle t^{m'+1}\rangle \rightarrow \Spec k[t]/\langle t^{m+1}\rangle$.
	\end{center}
For any $k$-scheme $Z$,
this induces a morphism
	\begin{center}
		$\varphi(Z) : Z \times_{\Spec k} \Spec k[t]/\langle t^{m'+1}\rangle \rightarrow Z \times_{\Spec k} \Spec k[t]/\langle t^{m+1}\rangle$.
	\end{center}
Then the collection of the morphisms $\varphi(Z)$ induces a natural transformation $F_m^X \rightarrow F_{m'}^X$ given by
	\begin{center}
		$F_m^X(Z) \rightarrow F_{m'}^X(Z)\ ;\ g \mapsto g \circ \varphi(Z)$.
	\end{center}
Since $F_m^X$ and $F_{m'}^X$ are represented by $X_m$ and $X_{m'}$,
we obtain a morphism
	\begin{center}
		$\pi_{m,m'} : X_m \rightarrow X_{m'}$.
	\end{center}
	\begin{Def}
The morphism $\pi_{m,m'}$ is called the \emph{truncation morphism}.
In particular,
for $m' = 0$,
we denote $\pi_{m,0}$ by $\pi_m$.
	\end{Def}

Let us look at the truncation morphism in the case $X = \mathbb{A}^e$.
Let $m$ and $m'$ be nonnegative integers with $m > m'$.
A closed point of $X_m$ can be written as $\mathbf{a} = (a_1^{(0)},...,a_1^{(m')},...,a_1^{(m)},...,$
$a_e^{(0)},...,a_e^{(m')},...,a_e^{(m)})$.
Then
	\begin{center}
$\pi_{m,m'}(\mathbf{a}) = (a_1^{(0)},...,a_1^{(m')},...,a_e^{(0)},...,a_e^{(m')})$.
	\end{center}
For a closed subscheme $X \subseteq \mathbb{A}^e$,
the truncation morphism for $X$ is the restriction of the truncation morphism for $\mathbb{A}^e$.

	\section{Intersections of irreducible components of the singular fiber of a jet scheme : $A_n$ case}

In this section,
we consider a surface $X$ over $\CC$ with an $A_n$-type singularity at the origin.
H. Mourtada studied irreducible components of the ``singular fiber"  of jet schemes of $X$.
His article(\cite{M1}) gives an explicit description of the defining ideals of the irreducible components.
We will first summarizes his arguments in a form convenient for us.

Let $X = \Spec \CC[x,y,z]/\langle xy-z^{n+1}\rangle$,
$m \in \mathbb{Z}_{\geq 0}$ and let $\pi_{m} : X_m \rightarrow X$ be the truncation morphism.
The surface $X$ is a toric surface which has an $A_n$-type singular point at $0(x = y = z = 0)$.
We are interested in the fiber $X_m^0 := \pi^{-1}_m(0)$,
which we call the singular fiber.
It is known that the number of irreducible components of $X_m^0$ is equal to the number of the exceptional curves of the minimal resolution of $X$ for $m \geq n$ (\cite[Theorem 3.1]{M1}).
Thus we assume $m \geq n$ from now on in this section.
We also note that if $I \subset \CC[x_0,..x_N,y_0,...,y_N,z_0,...,z_N]$ is an ideal for $N \in \mathbb{Z}_{\geq 0}$,
then for $N' \geq N$,
	\begin{center}
$I \cdot \CC[x_0,..x_{N'},y_0,...,y_{N'},z_0,...,z_{N'}] \cap \CC[x_0,..x_N,y_0,...,y_N,z_0,...,z_N] = I$.
	\end{center}
Hence we may regard the ideals appearing in the following as ideals in the ring $\CC[x_0,...,x_N,$ $y_0,...,y_N,z_0,...,z_N]$ for $N \gg 0$.

Now we fix some notations.
	\begin{Nota} \label{jet polynomial modulo coordinate}
Let $f := xy - z^{n+1}$ and let $p,q,r,m$ be nonnegative integers with $\mathrm{max}\{p,q,r\} \leq m$.
We define an ideal $L_{pqr}$ by
	\begin{center}
$L_{pqr} := \langle x_0,...,x_{p-1},y_0,...,y_{q-1},z_0,...,z_{r-1} \rangle$.
	\end{center}
For $j \leq m$,
we denote by $f_{pqr}^{(j)}$ the coefficient of $t^j$ in the expansion of $\displaystyle f(\sum_{i = p}^m x_it^i, \sum_{i=q}^m y_it^i, \sum_{i=r}^m z_it^i)$.

Moreover,
we set
	\begin{center}
$\Lambda_{pq}^{j} := \{ (l_1, l_2) \in \mathbb{Z}_{\geq 0}^2 \mid l_1 \geq p, l_2 \geq q\ \mathrm{and}\ l_1 + l_2 = j \}$
	\end{center}
and
	\[
\Lambda_{r}^j := \left \{ \left((i_1,...,i_l), (d_1,...,d_l)\right) \mid
		\begin{array}{l}
		\text{$l \geq 1, r \leq i_1 < i_2 < \cdots < i_l \leq j,\ d_1,...,d_l > 0,$}\\
		\text{$d_1 + \cdots + d_l = n+1,\ i_1d_1 + \cdots + i_ld_l = j$}
		\end{array}  \right \}.
	\]
	\end{Nota}
We obtain the following lemma by a direct calculation.
	\begin{Lemma} \label{An polynomial modulo relation}
For $p,q,r,j \in \mathbb{Z}_{\geq0}$ with $\mathrm{max} \{p,q,r \} \leq m$ and $j \leq m$,
we have
	\begin{center}
$f_{pqr}^{(j)} \equiv f^{(j)}\ \mathrm{mod}\ L_{pqr}$.
	\end{center}
In particular,
$f_{000}^{(j)} = f^{(j)}$.

Moreover,
we have
	\begin{center}
$\displaystyle f_{pqr}^{(j)} = \sum_{(l_1,l_2) \in \Lambda_{pq}^j} x_{l_1}y_{l_2} - \sum_{\left((i_1,...,i_l), (d_1,...,d_l) \right) \in \Lambda_{r}^j} \frac{(n+1)!}{d_1!\cdots d_l!} z_{i_1}^{d_1} \cdots z_{i_l}^{d_l}$,
	\end{center}
where the first (resp. second) term of the right hand side is $0$ if $\Lambda_{pq}^{j} = \emptyset$ (resp. $\Lambda_{r}^j = \emptyset$).
	\end{Lemma}
	\begin{Coro}(\cite[Section 3]{M1}) \label{Properties of jet polynomials}
	\begin{itemize}
\item[(1)] If $p + q > j$,
then
	\begin{center}
$\displaystyle f^{(j)}_{pqr} = - \sum_{\left((i_1,...,i_l), (d_1,...,d_l) \right) \in \Lambda_{r}^j} \frac{(n+1)!}{d_1!\cdots d_l!} z_{i_1}^{d_1} \cdots z_{i_l}^{d_l}$.
	\end{center}
\item[(2)] If $r(n+1) > j$,
then
	\begin{center}
$\displaystyle f^{(j)}_{pqr} = \sum_{(l_1,l_2) \in \Lambda_{pq}^j} x_{l_1}y_{l_2}$.
	\end{center}
\item[(3)] If $p+q > j$ and $r(n+1) > j$,
then
	\begin{center}
$f^{(j)}_{pqr} = 0$.
	\end{center}
	\end{itemize}
	\end{Coro}
	\begin{Nota} \label{periodicity of polynomial f}
We define the polynomial $g_{l,e}^{(j)}$ by
	\begin{center}
$g_{l,e}^{(j)} := f^{(j)}(x_l,...,x_{l+j},y_{e(n+1)-l},...,y_{e(n+1)-l+j},z_e,...,z_{e+j})$,
	\end{center}
where $f^{(j)}$ is regarded as an element of $\CC[x_0,...,x_j,y_0,...,y_j$ $,z_0,...,z_j]$ according to Remark \ref{jet polynomial of order m can be calculate as more than the order m variables}.
	\end{Nota}
Since the polynomial $f$ is weighted homogeneous,
we have the following useful fact.
	\begin{Lemma} (\cite[Section 3]{M1}when $e = 1$) \label{toCalculateJetPolynomialsWhenSomeCoordinate=0}
Assume $e,j,l,n \in \mathbb{Z}_{\geq 0}$ with $0 \leq l \leq e(n+1)$ and $j \geq 0$.
We have
	\begin{alignat}{2}
f_{l,e(n+1) -l,e}^{(e(n+1)+j)} &= g_{l,e}^{(j)}. 
	\end{alignat}
	\end{Lemma}
	\begin{proof}
For $0 \leq l \leq e(n+1)$,
we can calculate as follows:
	\begin{alignat*}{4}
\displaystyle
\ &f\left(\sum_{i=l}^{m}x_{i}t^i, \sum_{i=e(n+1)-l}^{m}y_{i}t^i, \sum_{i=e}^{m}z_{i}t^i \right)\\
=\ &
\displaystyle
f\left(t^l\sum_{i=0}^{m-l}x_{l+i}t^i,t^{e(n+1)-l}\sum_{i=0}^{m-e(n+1)+l}y_{e(n+1)-l+i}t^{i},t^{e}\sum_{i=0}^{m-e}z_{e+i}t^{i} \right)\\
 =\ &
\displaystyle
t^{e(n+1)}f\left(\sum_{i=0}^{m-l}x_{l+i}t^{i},\sum_{i=0}^{m-e(n+1)+l}y_{e(n+1)-l+i}t^{i},\sum_{i=0}^{m-e}z_{e+i}t^{i} \right).\\
	\end{alignat*}
Then $m-l, m-e(n+1)+l, m-e \geq 0$,
we have
	\begin{alignat*}{7}
\displaystyle
\ &f\left(\sum_{i=0}^{m-l}x_{l+i}t^{i},\sum_{i=0}^{m-e(n+1)+l}y_{e(n+1)-l+i}t^{i},\sum_{i=0}^{m-e}z_{e+i}t^{i} \right)\\
 \equiv\ &
\displaystyle
\sum_{j=0}^{m-e(n+1)}f^{(j)}(x_{l},...,x_{l+j},y_{e(n+1)-l},...,y_{e(n+1)-l+j},z_{e},...,z_{e+j})t^{j}
=\ &
\sum_{j=0}^{m-e(n+1)} g_{l,e}^{(j)}t^j
	\end{alignat*}
modulo $t^{m-e(n+1)+1}$.
Hence we have
	\begin{alignat*}{10}
\displaystyle
\ &f\left(\sum_{i=l}^{m}x_{i}t^i, \sum_{i=e(n+1)-l}^{m}y_{i}t^i, \sum_{i=e}^{m}z_{i}t^i \right)
 \equiv\ &
 t^{e(n+1)}\sum_{j=0}^{m-e(n+1)} g_{l,e}^{(j)}t^j
 =\ &
 \sum_{j=0}^{m-e(n+1)} g_{l,e}^{(j)}t^{e(n+1)+j}
 	\end{alignat*}
modulo $t^{m+1}$.
Looking at the coefficients of $t^{e(n+1)+j}$ for $0 \leq j \leq m-e(n+1)$,
we have
	\begin{alignat*}{2}
f_{l,e(n+1)-l,e}^{(e(n+1)+j)} &= g_{l,e}^{(j)}.
	\end{alignat*}
	\end{proof}
	\begin{Rem}\label{Slide Polynomials and apeearing coordinates}
We note that the variables appearing in $g_{l,e}^{(j)}$ are disjoint from $x_0,...,x_{l-1}, y_0,...$, $y_{e(n+1)-l -1}$ and $z_0,...,z_{e-1}$.
	\end{Rem}
We will now describe the defining ideals of the irreducible components of the singular fiber.
	\begin{Nota} \label{Definition of the defining ideal of the irreducible components}
Let $l \in \mathbb{Z}$ with $1 \leq l \leq n$.
We define the ideal $G_{m}^{l}$ by
	\begin{center}
$G^l_{m} = \langle g_{l,1}^{(0)},...,g_{l,1}^{(m-n-1)} \rangle$.
	\end{center}
(For $m = n$,
we set $G_n^l = 0$.)
We define the ideal $I_m^l$ by
	\begin{center}
$I_m^l := \langle L_{l,n+1-l,1}, f^{(0)},..., f^{(m)} \rangle = \langle x_0,...,x_{l-1},y_0,...,y_{n-l},z_0,z_1,f^{(0)},...,f^{(m)} \rangle$.
	\end{center}
Let $Z_m^l$ denote the subvariety of $(\mathbb{A}^3)_m \cong \mathbb{A}^{3(m+1)}$ defined by $I_m^l$,
i.e.
	\begin{center}
$Z_m^l = \mathbf{V}(I_m^l)$.
	\end{center}
	\end{Nota}
	\begin{Lemma}(\cite[Section 3]{M1}) \label{standerd Notation of Ideals}
Let $l \in \mathbb{Z}$ with $1 \leq l \leq n$.
We have
	\begin{center}
$I_m^l = L_{l,n+1-l,1} + G_m^l$.
	\end{center}
In particular,
if $m = n$,
then we have
	\begin{center}
$I_n^l = L_{l,n+1-l,1}$.
	\end{center}
	\end{Lemma}
	\begin{proof}
We apply Corollary \ref{Properties of jet polynomials} to $f_{l, n+1-l, 1}^{(i)}$ for $i = 0,..., n$.
From the assumption,
we have $l + (n+1-l) = n+1 > i$ and $1\times (n+1) = n+1 > i$,
and hence we have $f_{l,n+1-l,1}^{(i)} = 0$ by Corollary \ref{Properties of jet polynomials}(3).
By Lemma \ref{An polynomial modulo relation},
we have
	\begin{center}
$f^{(i)} \equiv f_{l,n+1-l,1}^{(i)} = 0 \ \mathrm{mod}\ L_{l,n+1-l,1}$
	\end{center}
for $i = 0, ..., n$.
Moreover,
from Lemma \ref{An polynomial modulo relation},
we have $f^{(n+1+j)} \equiv f^{(n+1+j)}_{l,n+1-l,1}$ mod $L_{l,n+1-l,1}$,
if $n+1+j \leq m$,
and from Lemma \ref{toCalculateJetPolynomialsWhenSomeCoordinate=0},
we have
	\begin{center}
$g^{(j)}_{l,1} = f^{(n+1+j)}_{l,n+1-l,1} \equiv f^{(n+1+j)} \ \mathrm{mod} \ L_{l,n+1-l,1}$.
	\end{center}
Hence in the case $m = n$,
we have
	\begin{center}
$I_n^l = L_{l,n+1-l,1}$,
	\end{center}
and in the case $m > n$,
we have
	\begin{alignat}{3}
I_m^l =&\  L_{l,n+1-l,1} + \langle f^{(n+1)},...,f^{(m)} \rangle \\
 =&\  L_{l,n+1-l,1} + \langle g_{l,1}^{(0)},...,g_{l,1}^{(m-n-1)} \rangle &\ = L_{l,n+1-l,1} + G_m^l.
	\end{alignat}
	\end{proof}

These closed subvarieties give the irreducible decomposition of $X_m^0$.

	\begin{Prop} (\cite[Theorem 3.1]{M1}, \cite[Proposition 1.5, Theorem 3.3]{Mu}) \label{irreducibility and irreducible decomposition of singular fiber}
We have
	\begin{center}
$Z_n^l \cong \mathbb{A}^{2n+1}$
	\end{center}
and
	\begin{center}
$Z_m^l \cong X_{m-n-1} \times \mathbb{A}^{2n+1}$
	\end{center}
if $m \geq n+1$.
Moreover,
the ideal $\langle f^{(0)},...,f^{(l)} \rangle$ which is the defining of $X_l$ is prime for $l \geq 0$,
so the variety $X_l$ is irreducible.
In particular,
for $m \geq n$,
the varieties $Z_m^l$ are irreducible.
The irreducible decomposition of $X_m^0$ is given by
	\begin{center}
$\displaystyle X_m^0 = \bigcup_{l=1}^n Z_m^l$.
	\end{center}
	\end{Prop}

Now we study the intersections of irreducible components of $X_m^{0}$.
We define the following ideals:
For $1 \leq i < j \leq n$,
	\begin{center}
$J_m^{i,j} := I_m^i + I_m^j$\ ,\ \ $I_m^{i,j}$ := $\sqrt{J_m^{i,j}}$.
	\end{center}
Recalling Notation \ref{Definition of the defining ideal of the irreducible components},
we have 
	\begin{alignat}{5}
L_{j,n+1-i,1} =&\ \langle x_0,...,x_{j-1},y_0,...,y_{n-i},z_0 \rangle \\
=&\ L_{i,n+1-i,1} + L_{j,n+1-j,1}
	\end{alignat}
and have
	\begin{alignat}{3}
J_m^{i,j} = L_{j,n+1-i,1} + \langle f^{(0)},...,f^{(m)} \rangle.
	\end{alignat}
From the definition of $I_m^{i,j}$,
we have $\mathbf{V}(I_m^{i,j}) = Z_m^i \cap Z_m^j$.
Now we give the irreducible decomposition of the closed subvariety $Z_m^i \cap Z_m^j$.

	\begin{Thm} \label{irreducible decomposition of intersection of A_n-type}
Assume $m \geq n \geq 2$.
Let $1 \leq i < j \leq n$.
	\begin{itemize}
\item[(a)] If $m = n$,
then $I_n^{i,j} = L_{j,n+1-i,1}$ and
	\begin{center}
$Z_n^i \cap Z_n^j = \mathbf{V}(L_{j,n+1-i,1})$
	\end{center}
is irreducible.
\item[(b)] If $1 \leq m-n \leq j-i$,
then $I_m^{i,j} = L_{j,n+1-i,2}$ and
	\begin{center}
$Z_m^i \cap Z_m^j = \mathbf{V}(L_{j,n+1-i,2})$
	\end{center}
is irreducible.
\item[(c)] If $j-i \leq m-n$ and $m < 2n+2$,
then
	\begin{center}
$\displaystyle I_m^{i,j} = \bigcap_{u=0}^{m-n-(j-i)} L_{j+u,m-j-u+1,2}$
	\end{center}
and hence the irreducible decomposition of $Z_m^i \cap Z_m^j$ is given by
	\begin{center}
$Z_m^i \cap Z_m^j = \displaystyle  \bigcup_{u = 0}^{m - n-(j-i)} \mathbf{V}(L_{j+u,m-j-u+1,2})$.
	\end{center}
\item[(d)] If $m \geq 2n + 2$,
then
	\begin{center}
$\displaystyle I_m^{i,j} = \bigcap_{u=0}^{n+1-(j-i)} (L_{j+u,2n+2 -j-u,2} + \langle f^{(2n+2)},...,f^{(m)} \rangle)$.
	\end{center}
The ideal $L_{j+u,2n+2 -j-u,2} + \langle f^{(2n+2)},...,f^{(m)} \rangle$ is prime for $0 \leq u \leq n+1-(j-i)$,
and the irreducible decomposition of $Z_m^i \cap Z_m^j$ is given by
	\begin{center}
$Z_m^i \cap Z_m^j = \displaystyle  \bigcup_{u = 0}^{n+1 - (j-i)} \mathbf{V}(L_{j+u,2n+2 -j-u,2} + \langle f^{(2n+2)},...,f^{(m)} \rangle)$.
	\end{center}
	\end{itemize}

	\end{Thm}
	\begin{Rem} \label{the index conditions in thm}
In (b),
we have $1 \leq m-n \leq j-i$,
so $n < m \leq n+(j-i) \leq 2n-1 < 2n+2$.

In (c),
the number of irreducible components of $Z_m^i \cap Z_m^j$ is $m-n-(j-i)+1 \geq 1$.

In (d),
the number of irreducible components of $Z_m^i \cap Z_m^j$ is $n-(j-i)+2 \geq 3$.

In particular,
in the cases (c) and (d),
the number of irreducible components decreases as $j-i$ increases.
	\end{Rem}
	\begin{proof}
First of all,
from Lemma \ref{An polynomial modulo relation}, we have
	\begin{center}
$f^{(l)} \equiv f_{j,n+1-i,1}^{(l)}\ \mathrm{mod}\ L_{j,n+1-i,1}$.
	\end{center}
(a)
Recall that $J_m^{i,j} = L_{j,n+1-i,1} + \langle f^{(0)},...,f^{(m)} \rangle$.
We prove that
	\begin{center}
$J_n^{i,j} = L_{j,n+1-i,1} = \langle x_0,...,x_{j-1},y_0,...,y_{n-i},z_0 \rangle$,
	\end{center}
i.e.,
$f^{(0)},...,f^{(n)} \in L_{j,n+1-i,1}$.
In fact,
for $l = 0,...,n$,
by Corollary \ref{Properties of jet polynomials}(3) and $j + (n+1-i) > n \geq l$ and $1\times (n+1) > n \geq l$,
we have
	\begin{center}
$f_{j,n+1-i,1}^{(l)} =  0$.
	\end{center}
So we have
	\begin{center}
$f^{(l)} \equiv f_{j,n+1-i,1}^{(l)} =  0$ mod $L_{j,n+1-i,1}$
	\end{center}
for $l = 0,...,n$ and
	\begin{center}
$J_n^{i,j} = L_{j,n+1-i,1} = \langle x_0,...,x_{j-1},y_0,...,y_{n-i},z_0 \rangle$.
	\end{center}
This ideal is clearly prime,
hence
	\begin{center}
$I_n^{i,j} = L_{j,n+1-i,1}$.
	\end{center}
(b)
We prove that
	\begin{center}
$I_m^{i,j} = L_{j,n+1-i,2} = L_{j,n+1-i,1} + \langle z_1 \rangle$.
	\end{center}
In fact,
from Corollary \ref{Properties of jet polynomials}(1) and $j + (n-i+1) > n+1$,
we have
	\begin{center}
$f^{(n+1)}_{j,n-i+1,1} = -z_1^{n+1}$.
	\end{center}
It follows that
	\begin{center}
$J_{n+1}^{i,j} = L_{j,n+1-i,1} + \langle z_1^{n+1} \rangle$,
	\end{center}
hence we have
	\begin{center}
$I_{n+1}^{i,j} = \sqrt{J_{n+1}^{i,j}} = L_{j,n+1-i,2}$,
	\end{center}
and this ideal is clearly prime.
For any $l$ with $n+1 < l \leq m$,
we have $j + (n+1-i) > m \geq l$ from the assumption $j-i \leq m-n$ and $2(n+1) > m \geq l$ and by Remark \ref{the index conditions in thm}.
Hence we have
	\begin{center}
$f^{(l)}_{j,n+1-i,2} = 0$,
	\end{center}
by Corollary \ref{Properties of jet polynomials}(3),
and
	\begin{center}
$f^{(l)} \equiv 0\ \mathrm{mod}\ L_{j,n+1-i,2}$ $\ (l = n+2,...,m)$
	\end{center}
by Lemma \ref{An polynomial modulo relation}.
Thus
	\begin{center}
$I_{n+1}^{i,j} + \langle f^{(n+2)},...,f^{(m)} \rangle = L_{j,n+1-i,2} + \langle f^{(n+2)},...,f^{(m)} \rangle = L_{j,n+1-i,2}$.
	\end{center}
holds.
We also have
	\begin{center}
$J_{m}^{i,j} \subseteq I_{n+1}^{i,j} + \langle f^{(n+2)},...,f^{(m)} \rangle = L_{j,n+1-i,2} \subseteq I_m^{i,j}$,
	\end{center}
and taking the radicals,
we have $I_m^{i,j} = L_{j,n+1-i,2}$.\\
(c)
In general,
for a subvariety $V = \mathbf{V}(I) \subseteq X_m$,
we have $\pi_{m+1,m}^{-1}(V) = \mathbf{V}(I + \langle f^{(m+1)} \rangle)$.
Hence $\pi_{m+1,m}^{-1}(Z_m^i) = Z_{m+1}^i$.
If $Z_m^i \cap Z_m^j = V_1\cup \cdots \cup V_r$,
then
	\begin{center}
$Z_{m+1}^i \cap Z_{m+1}^j = \pi_{m+1,m}^{-1}(Z_m^i) \cap \pi_{m+1,m}^{-1}(Z_m^j) = \pi_{m+1,m}^{-1}(V_1) \cup \cdots \cup \pi_{m+1,m}^{-1}(V_r)$.
	\end{center}
Using this we prove the assertion by induction on $m \geq n+ (j-i)$.

The case $m = n + (j-i)$ actually belongs to the case (b) since $m-n = j-i$,
and the assertion is true.

For $m > n + (j-i)$,
we assume that the claim is true for $m-1$,
i.e.
the irreducible decomposition of $Z_{m-1}^i \cap Z_{m-1}^j$ is
	\begin{center}
$Z_{m-1}^i \cap Z_{m-1}^j = \displaystyle  \bigcup_{u = 0}^{m-1 - n-(j-i) } \mathbf{V}(L_{j+u,m-j-u,2})$.
	\end{center}
Now,
we consider the ideal
	\begin{center}
$\langle x_0,...,x_{j-1 + u},y_0,...,y_{m-1 -j-u},z_0,z_1,f^{(m)} \rangle = L_{j+u,m-j-u,2} + \langle f^{(m)} \rangle$,
	\end{center}
which is a defining ideal of $\pi_{m,m-1}^{-1}(\mathbf{V}(L_{j+u,m-j-u,2}))$.
We apply Corollary \ref{Properties of jet polynomials}(2) to $f^{(m)}_{j+u,m-j-u,2}$.
We have $2(n+1) > m$ from the assumption of the assertion (c),
and so from $(j+u) + (m-j-u) = m$,
we have
	\begin{center}
$f^{(m)}_{j+u,m-j-u,2} = x_{j+u}y_{m-j-u}$.
	\end{center}
By Lemma \ref{An polynomial modulo relation},
we have
	\begin{center}
$f^{(m)} \equiv f^{(m)}_{j+u,m-j-u,2}\ \mathrm{mod}\ L_{j+u,m-j-u,2}$.
	\end{center}
Hence we have
	\begin{alignat}{5}
L_{j+u,m-j-u,2} + \langle f^{(m)} \rangle &\ = L_{j+u,m-j-u,2} + \langle x_{j+u}y_{m-j-u} \rangle\\
 &\ =\ (L_{j+u,m-j-u,2} + \langle x_{j+u} \rangle) \cap (L_{j+u,m-j-u,2} + \langle y_{m-j-u} \rangle)\\
&\ =\ L_{j+u+1,m-j-u,2} \cap L_{j+u,m-j-u+1,2}.
	\end{alignat}
The two ideals in the right hand side are prime,
so the irreducible decomposition of $\pi_{m,m-1}^{-1}(\mathbf{V}(L_{j+u,m-j-u,2}))$ is given by
	\begin{center}
$\pi_{m,m-1}^{-1}(\mathbf{V}(L_{j+u,m-j-u,2})) = \mathbf{V}(L_{j+u+1,m-j-u,2}) \cup \mathbf{V}(L_{j+u,m-j-u+1,2})$.
	\end{center}
Thus we have
	\begin{center}
$Z_m^i \cap Z_m^j = \displaystyle  \bigcup_{u = 0}^{m - n-(j-i)} \mathbf{V}(L_{j+u,m-j-u+1,2})$.
	\end{center}
Finally,
we have to prove that none of the ideals $L_{j+u,m-j-u+1,2}$ contain another.
In fact,
for $u_1 < u_2$,
we have
	\begin{align}
x_{j + u_2 -1} \in&\ L_{j+u_2,m-j-u_2+1,2} - L_{j+u_1,m-j-u_1+1,2}\\
\mathrm{and}\ y_{m-j-u_1} \in&\ L_{j+u_1,m-j-u_1+1,2} - L_{j+u_2,m-j-u_2+1,2}.
	\end{align}
So the above decomposition is the irreducible decomposition.
\\
(d)
As in the case (c),
we have only to show that the ideal
	\begin{center}
$L_{j+u,2n+2-j-u,2} + \langle f^{(2n+2)},...,f^{(m)} \rangle = \langle x_0,...,x_{j-1 + u},y_0,...,y_{2n-j+1-u},z_0,z_1,f^{(2n+2)},...,f^{(m)} \rangle$,
	\end{center}
which is a defining ideal of $\pi_{m,2n+1}^{-1}(\mathbf{V}(L_{j+u,2n+2-j-u,2})),$
is prime for $u = 0,...,n+1-(j-i)$,
and none of the ideals $L_{j+u,2n+2-j-u,2} + \langle f^{(2n+2)},..., f^{(m)} \rangle$ contain another.

First we prove that the ideal
	\begin{center}
$L_{j+u,2n+2-j-u,2} + \langle f^{(2n+2)},..., f^{(m)} \rangle$
	\end{center}
is prime.
From Lemma \ref{An polynomial modulo relation} and Lemma \ref{toCalculateJetPolynomialsWhenSomeCoordinate=0},
we have
	\begin{alignat}{4}
f^{(2n+2+v)} &\ \equiv \ && f^{(2n+2+v)}_{j+u,2n+2-j-u,2} &&\ \mathrm{mod}\ L_{j+u,2n+2-j-u,2}\\
&\ =\ && g_{j+u,2}^{(v)}.
	\end{alignat}
for $v = 0,...,m - 2(n+1)$.
Then
	\begin{center}
$L_{j+u,2n+2-j-u,2} + \langle f^{(2n+2)},...,f^{(m)} \rangle = L_{j+u,2n+2-j-u,2} + \langle g^{(0)}_{j+u,2},...,g^{(m - 2n-2)}_{j+u,2} \rangle$.
	\end{center}
From Proposition \ref{irreducibility and irreducible decomposition of singular fiber},
the ideal $\langle g^{(0)}_{j+u,2},...,g^{(m - 2n-2)}_{j+u,2} \rangle$ is prime.
Moreover,
we have
	\begin{alignat}{3}
g^{(v)}_{j+u,2} &\ \in k[x_{j+u},...,x_{j+u+v},y_{2n+2-j-u},...,y_{2n+2-j-u+v},z_2,...,z_{2+v}]\\
 &\ \subseteq k[x_{j+u},...,x_{j+u+m-2n-2},y_{2n+2-j-u},...,y_{m-j-u},z_2,...,z_{m-2n}]
	\end{alignat}
for $v = 0,...,m - 2(n+1)$ from Remark \ref{Slide Polynomials and apeearing coordinates}.
In particular,
the variables appearing in $g^{(v)}_{j+u,2}$ are disjoint from $x_0,...,x_{j-1 + u},y_0,...,y_{2n-j+1-u}$ and $z_0, z_1$.
So the ideal
	\begin{center}
$L_{j+u,2n+2-j-u,2} + \langle g^{(0)}_{j+u,2},...,g^{(m - 2n-2)}_{j+u,2} \rangle$
	\end{center}
is also prime.
Hence the variety
	\begin{center}
$\mathbf{V}(L_{j+u,2n+2-j-u,2} + \langle g^{(0)}_{j+u,2},...,g^{(m - 2n-2)}_{j+u,2} \rangle) \cong X_{m-2(n+1)}\times \mathbb{A}^{4n+2}$
	\end{center}
in $\mathbb{A}^{3m+3}$ is irreducible.

Next we check for $u_1,u_2 \in \{0,...,n+1-(j-i)\}$ with $u_1 < u_2$,
that the prime ideals $L_{j+u_l,2n+2-j-u_l,2} + \langle g^{(0)}_{j+u_l,2},...,g^{(m - 2n-2)}_{j+u_l,2} \rangle$ ($l = 1,2$) are not contained in each other.
We have $x_{j + u_2 -1} \notin \langle g^{(0)}_{j+u_1,2},...,g^{(m - 2n-2)}_{j+u_1,2} \rangle$ since the ideal $\langle g^{(0)}_{j+u_1,2},...,g^{(m - 2n-2)}_{j+u_1,2} \rangle$ is reduced and the degree of all terms of the generators are at least $2$.
So we have
	\begin{alignat}{3}
&\ x_{j + u_2 -1} \notin&\ L_{j+u_1,2n+2-j-u_1,2} + \langle g^{(0)}_{j+u_1,2},...,g^{(m - 2n-2)}_{j+u_1,2} \rangle\\
\mathrm{and}\ &\ x_{j + u_2 -1} \in&\ L_{j+u_2,2n+2-j-u_2,2} + \langle g^{(0)}_{j+u_2,2},...,g^{(m - 2n-2)}_{j+u_2,2} \rangle.
	\end{alignat}
Similarly,
we can check that
	\begin{alignat}{3}
&\ y_{2n+1-j-u_1} \notin&\ L_{j+u_2,2n+2-j-u_2,2} + \langle g^{(0)}_{j+u_2,2},...,g^{(m - 2n-2)}_{j+u_2,2} \rangle\\
\mathrm{and}\ &\ y_{2n+1-j-u_1} \in&\ L_{j+u_1,2n+2-j-u_1,2} + \langle g^{(0)}_{j+u_1,2},...,g^{(m - 2n-2)}_{j+u_1,2} \rangle.
	\end{alignat}
Hence the decomposition
	\begin{center}
$Z_m^i \cap Z_m^j = \displaystyle  \bigcup_{u = 0}^{n+1 - (j-i)} \mathbf{V}(L_{j+u,2n+2-j-u,2} + \langle f^{(2n+2)},...,f^{(m)} \rangle)$
	\end{center}
is an irreducible decomposition.
	\end{proof}
	\begin{Coro} \label{dimensions of A_n-type}
Assume $m \geq n \geq 2$ and $1 \leq i < j \leq n$.
The intersections $Z_m^i \cap Z_m^j$ of irreducible components of singular fiber are of pure dimension,
and the following hold.
	\begin{itemize}
\item[(a)] If $m=n$,
then $\mathrm{dim}\ Z_n^i \cap Z_n^j = 2n - (j - i) +1$,
\item[(b)] Assume $m > n$.
		\begin{itemize}
	\item[{\rm(\hspace{.18em}i\hspace{.18em})}] If $m-n < j-i$,
	then $\mathrm{dim}\ Z_m^i \cap Z_m^j = 3m - n - (j - i)$,
	\item[{\rm(\hspace{.18em}ii\hspace{.18em})}] If $j-i \leq m-n$,
	then $\mathrm{dim}\ Z_m^i \cap Z_m^j = 2m$.
		\end{itemize}
	\end{itemize}
	\end{Coro}
	\begin{proof}
(a)
From Proposition \ref{irreducible decomposition of intersection of A_n-type}(a),
the closed subvariety $Z_n^i \cap Z_n^j$ is defined by the ideal $I_n^{i,j} = L_{j,n+1-i,1} = \langle x_0,...,x_{j-1},y_0,...,y_{n-i},z_0 \rangle$ in $\mathbb{A}^{3(n+1)}$.
So
	\begin{center}
$\mathrm{dim}\ Z_n^i \cap Z_n^j = 3(n+1) - j - (n+1-i) -1 = 2n - (j-i) +1$.
	\end{center}
(b) {\rm(\hspace{.18em}i\hspace{.18em})}
From Proposition \ref{irreducible decomposition of intersection of A_n-type}(b),
the closed subvariety $Z_m^i \cap Z_m^j$ is defined by the ideal $I_m^{i,j} = L_{j,n-i+1,2} = \langle x_0,...,x_{j-1},y_0,...,y_{n-i},z_0,z_1 \rangle$ in $\mathbb{A}^{3(m+1)}$.
So
	\begin{center}
$\mathrm{dim}\ Z_m^i \cap Z_m^j = 3(m+1) - j - (n+1-i) -2 = 3m - n - (j-i)$.
	\end{center}
{\rm(\hspace{.18em}ii\hspace{.18em})}
From Proposition \ref{irreducible decomposition of intersection of A_n-type}(c) and (d),
the irreducible component of $Z_m^i \cap Z_m^j$ is isomorphic to $X_{m-2(n+1)} \times \mathbb{A}^{4n+2}$.
Thus
	\begin{center}
$\mathrm{dim}\ Z_m^i \cap Z_m^j = \mathrm{dim}\ X_{m-2(n+1)} \times \mathbb{A}^{4n+2} = 2(m-2n-1) + 4n+2 = 2m$.
	\end{center}
	\end{proof}
	\begin{Ex}
For an $A_3$-type singular surface,
Table \ref{table : dimensins and codimnsions} shows the dimensions and  codimensions of $Z_m^i$ and $Z_m^i \cap Z_m^j$ and the numbers of the irreducible components $N_m^{ij}$ of $Z_m^i \cap Z_m^j$ for small values of $m$.
	\begin{table}[htp]
		\caption{}
		\label{table : dimensins and codimnsions}
		\begin{tabular}{|c|c|c|c|c|c|c|c|c|c|}
			\hline
			 & \multicolumn{3}{|c|}{$\mathrm{dim}$} & \multicolumn{3}{|c|}{$\mathrm{codim}_{\mathbb{A}^{3(m+1)}}$} & \multicolumn{2}{|c|}{} \\
			\hline
			$m$ &\ \ $Z_m^i$\ \ & $Z_m^1 \cap Z_m^2$ & $Z_m^1 \cap Z_m^3$ &
			\ \ $Z_m^i$\ \ &
			$Z_m^1 \cap Z_m^2$ & $Z_m^1 \cap Z_m^3$ &\ \ $N_m^{12}$\ \ &\ \ $N_m^{13}$\ \ \\
			\hline
			3 & 7 & 6 & 5 & 5 & 6 & 7 & 1 & 1\\
			4 & 9 & 8 & 7 & 6 & 7 & 8 & 1 & 1\\
			5 & 11 & 10 & 10 & 7 & 8 & 8 & 2 & 1\\
			6 & 13 & 12 & 12 & 8 & 9 & 9 & 3 & 2\\
			7 & 15 & 14 & 14 & 9 & 10 & 10 & 4 & 3\\
			\hline
		\end{tabular}
	\end{table}
	\end{Ex}
	\begin{Coro} \label{maximal condition of A_n}
For $m \geq n$ and $i, j, k, l \in \{1,...,n\}$ with $i < j$ and $k < l$,
we have
	\begin{center}
$Z_m^i \cap Z_m^j \subset Z_m^k \cap Z_m^l$
	\end{center}
if and only if $i \leq k < l \leq j$.
In particular,
the intersection $Z_m^i \cap Z_m^j$ of the irreducible components of $X_m^0$ is maximal in $\{Z_m^i \cap Z_m^j \mid  i \neq j \}$ for the inclusion relation if and only if $| i - j | = 1$.
	\end{Coro}

Now we define a graph $\Gamma$ using the information on $X_m^0$ as follows.
	\begin{Const}\label{ConstructionofGraph}
The graph $\Gamma$ is constructed pairs $(V, E)$ as follows.
	\begin{itemize}
\item The vertices V are the irreducible components of $X_m^0$.
\item The edges E are all maximal elements of $\{ Z_m^i \cap Z_m^j \mid i \neq j \}$,
and $Z_m^i \cap Z_m^j$ connects $Z_m^i$ and $Z_m^j$.
	\end{itemize}
In other words,
an edge is given between $Z_m^i$ and $Z_m^j$ if and only if $Z_m^i \cap Z_m^j \in E$.
 	\end{Const}
	\begin{Coro}
The graph obtained by Construction \ref{ConstructionofGraph} is isomorphic to the resolution graph of an $A_n$-type singularity.
	\end{Coro}

	\section{Intersections of irreducible components of the singular fiber of a jet scheme : $D_4$ case}

From the result for $A_n$-type singular surfaces in the previous section,
we expect the following:
If $X$ is a surface over $\CC$ with a rational double point singularity,
the graph constructed in Construction \ref{ConstructiontheGraph} is isomorphic to the resolution graph of $X$ for $m \gg 0$.
In this section,
we show that this holds in the case $X$ has a $D_4$-type singular point.

In \cite{M2},
Mourtada describes how to obtain the irreducible components of the singular fiber of the jet scheme of $X$.
In this case,
unlike the case of $A_n$-type singularities,
we do not know the generators of the defining ideals of the irreducible components when the order $m$ is large.

Let $f(x,y,z) = x^2 - y^2z + z^3 \in \CC[x,y,z]$ and $X \subset \mathbb{A}^3$ the hypersurface defined by $f$.
For any $m \in \mathbb{Z}_{\geq 0}$,
let $R_m := \CC[x_0,...,x_m,y_0,...,y_m,z_0,...,z_m]$ and $\pi_m : X_m \rightarrow X$ be the truncation morphism.
We denote the origin $x = y = z = 0$ of $\mathbb{A}^3$ by $0$.
Then the surface $X$ has a singular point at $0$.
Let us denote the singular fiber $\pi^{-1}_m(0)$ by $X_m^0$.
We fix some more notations.

	\begin{Nota} \label{about Jet polynomials}
Assume $m \geq 5$.
We define a number of ideals of $R_m$ which are defining ideals of subvarieties of $X_m^0$:
	\begin{alignat}{1}
L_{pqr} = \langle x_0,...,x_{p-1},y_0,...,y_{q-1},z_0,...,z_{r-1} \rangle \label{mainPartI_m^0}
	\end{alignat}
as in section 2,
	\begin{alignat}{6}
L^1
\quad =&\quad \langle x_0,x_1, &&\, \; y_0, &&\, \; z_0,z_1 && &\rangle &\, \; = L_{212}, \label{mainPartI_m^1}\\
L^2
\quad =&\quad \langle x_0,x_1, &&\, \; y_0, &&\, \; z_0, &&\, \; y_1 - z_1 &\rangle, \label{mainPartI_m^2}\\
L^3
\quad =&\quad \langle x_0,x_1, &&\, \; y_0, &&\, \; z_0, &&\, \; y_1 + z_1 &\rangle, \label{mainPartI_m^3}
	\end{alignat}
and
	\begin{alignat}{5}
I_m^0
 &\ \ =\ \ &&L_{322} &\ + \langle f^{(0)},...,f^{(m)} \rangle, \label{PrimeIdealV_m^0}\\
J_m^1
 &\ \ =\ \ &&L^1 &\ + \langle f^{(0)},...,f^{(m)} \rangle, \label{OpenPrimeIdealV_m^1}\\
J_m^2
 &\ \ =\ \ &&L^2 &\ + \langle f^{(0)},...,f^{(m)} \rangle, \label{OpenPrimeIdealV_m^2}\\
J_m^3
 &\ \ =\ \ &&L^3 &\ + \langle f^{(0)},...,f^{(m)} \rangle. \label{OpenPrimeIdealV_m^3}
	\end{alignat}
	\end{Nota}
	\begin{Prop} (\cite[Section 3.2]{M2}) \label{J_m^i*(R_m)_{y_1} are not equal to (R_m)_{y_1}}
The ideal $I_m^0$ is a prime ideal in $R_m$,
and the ideals $J_m^i\cdot (R_m)_{y_1}$ for $i = 1,2,3$ are prime ideals in $(R_m)_{y_1}$.
In particular,
$\mathbf{V}(J_m^{i}) \cap \mathbf{D}(y_1)$ are irreducible subvarieties of $\mathbf{D}(y_1)$,
where $\mathbf{D}(y_1)$ is the open subscheme of $X_m$ defined by $y_1 \neq 0$.
	\end{Prop}
From these,
we define some irreducible subvarieties of $X_m$ contained in $X_m^0$.
 	\begin{Def} \label{definition of irreducible components of X_m^0 for D_4-type singular surface}
Let $Z_m^0,...,Z_m^3$ be the closed subvarieties of $X_m$,
contained in $X_m^0$,
defined by
	\begin{center}
$Z_m^0 := \mathbf{V}(I_m^0)$,
	\end{center}
and for $i \in \{1,2,3\}$,
	\begin{center}
$Z_m^i := \overline{\mathbf{V}(J_m^{i}) \cap \mathbf{D}(y_1)}$,
	\end{center}
where the bar means the Zariski closure in $X_m$.
	\end{Def}

The defining ideals of $Z_m^i$ are obtained as follows.
	\begin{Def} \label{definition of Z_m^i}
For $i = 1,2,3$,
we define the ideals $I_m^i$ by
	\begin{center}
$I_m^i := J_m^i\cdot (R_m)_{y_1} \cap R_m$.
	\end{center}
	\end{Def}
	\begin{Lemma} (\cite[Section 3.2]{M2}) \label{Zariski closure}
Assume $m \geq 5$.
The ideals $I_m^i$ are prime ideals in $R_m$ and $Z_m^i = \mathbf{V}(I_m^i)$.
	\end{Lemma}
 The singular fiber $X_m^0$ decomposes as follows.
 	\begin{Prop} (\cite[Section 3.2]{M2}) \label{Irreducible decomposition of singular fiber of D_4-type singular point}
Assume $m \geq 5$.
The irreducible decomposition of the singular fiber $X_m^0$ is given by
	\begin{center}
$X_m^0 = Z_m^0 \cup Z_m^1 \cup Z_m^2 \cup Z_m^3$.
	\end{center}
Moreover,
the dimensions of $Z_m^i$ are equal to $2m+1$ for $i = 0,...,3$.
	\end{Prop}
	\begin{Rem} (\cite[Theorem 3.2]{M2})
The codimension of $Z_m^i$ in $\mathbb{A}^{3(m+1)}$ is equal to $m+2$.
The defining ideal of $X_m^0$ is generated by $m+2$ elements $x_0,y_0,z_0$ and $f^{(2)},...,f^{(m)}$,
so $X_m^0$ is a complete intersection in $\mathbb{A}^{3(m+1)}$.
Note that the dimension of the smooth locus $X_{\mathrm{sm}}$ of $X$ is $2(m+1)$ and the codimension of $X_{\mathrm{sm}}$ in $\mathbb{A}^{3(m+1)}$ is $m+1$.
	\end{Rem}
From now on,
we suppose $m \geq 5$.

We will take advantage of the symmetries of $X$.
Let $\varphi_1$ and $\varphi_2$ be the automorphisms of $R_m$ defined by
	\[
\varphi_1 :
		\begin{cases}
	x_i \mapsto x_i, \\
	y_i \mapsto - y_i, \\
	z_i \mapsto z_i,
		\end{cases}
\varphi_2 :
		\begin{cases}
	x_i \mapsto x_i, \\
	y_i \mapsto -\tfrac{1}{2}y_i + \tfrac{3}{2}z_i, \\
	z_i \mapsto -\tfrac{1}{2}y_i - \tfrac{1}{2}z_i.
		\end{cases}
	\]
	\begin{Rem} \label{varphi preserve the f^{(i)}}
The automorphisms $\varphi_1$ and $\varphi_2$ are induced by the automorphisms of $\CC[x,y,z]$ defined by
	\[
		\begin{cases}
	x \mapsto x, \\
	y \mapsto -y, \\
	z \mapsto z,
		\end{cases}
\mathrm{and}
		\begin{cases}
	x \mapsto x, \\
	y \mapsto -\tfrac{1}{2}y + \tfrac{3}{2}z, \\
	z \mapsto -\tfrac{1}{2}y - \tfrac{1}{2}z,
		\end{cases}
	\]
for which the polynomial $f$ is invariant.
So $\varphi_1$ and $\varphi_2$ preserve the polynoimals $f^{(j)} (j \in \{0,...,m\})$.
Therefore,
$\varphi_1$ and $\varphi_2$ induce automorphisms of $X_m$.
We denote these morphisms by $\psi_1$ and $\psi_2$.
	\end{Rem}
Now we show how the morphisms $\psi_1$ and $\psi_2$ act on the set of the closed subvarieties $\{Z_m^0,Z_m^1, Z_m^2,Z_m^3\}$.
We need the following lemma.
	\begin{Lemma}\label{I_m^i another defining ideals}
Assume $m \geq 5$.
For $i = 1,2,3$,
we have
	\begin{center}
$Z_m^i = \overline{\mathbf{V}(J_m^i) \cap \mathbf{D}(y_1)} = \overline{\mathbf{V}(J_m^i) \cap \mathbf{D}(y_1 -3z_1)}$.
	\end{center}
	\end{Lemma}
	\begin{proof}
From Proposition \ref{J_m^i*(R_m)_{y_1} are not equal to (R_m)_{y_1}},
$y_1 \notin \sqrt{J_m^i}$ for $i = 1,2,3$.
Moreover,
we note that $z_1$,
$y_1 - z_1$ or $y_1 + z_1 \in J_m^i$.
Then $y_1(P) = 0$ if and only if $(y_1 - 3z_1)(P) =0$ for $P \in \mathbf{V}(J_m^i)$.
Thus $\mathbf{V}(J_m^i) \cap \mathbf{D}(y_1) = \mathbf{V}(J_m^i) \cap \mathbf{D}(y_1 -3z_1)$,
so
	\begin{center}
$\overline{\mathbf{V}(J_m^i) \cap \mathbf{D}(y_1 -3z_1)} = \overline{\mathbf{V}(J_m^i) \cap \mathbf{D}(y_1)} = Z_m^i$
	\end{center}
by Proposition \ref{J_m^i*(R_m)_{y_1} are not equal to (R_m)_{y_1}}.
	\end{proof}
	\begin{Prop} \label{trancefar irreducible components}
Assume $m \geq 5$.
The irreducible components of $X_m^0$ are mapped to another by $\psi_s (s = 1,2)$ as follows:
	\begin{enumerate}
\item[(a)] $\psi_1(Z_m^0) = Z_m^0$, $\psi_1(Z_m^1) = Z_m^1$, $\psi_1(Z_m^2) = Z_m^3$ and $\psi_1(Z_m^3) = Z_m^2$.
\item[(b)] $\psi_2(Z_m^0) = Z_m^0$, $\psi_2(Z_m^1) = Z_m^2$, $\psi_2(Z_m^2) = Z_m^3$ and $\psi_2(Z_m^3) = Z_m^1$.
	\end{enumerate}
	\end{Prop}
	\begin{proof}
The case $\psi_s(Z_m^0) = Z_m^0$ can be obtained by a direct calculation.
The other cases as follows.
For $i = 1,2,3$,
the subvariety $Z_m^i$ is irreducible,
and by the previous lemma we only have to show that
	\begin{center}
$\psi_s(\mathbf{V}(J_m^i) \cap \mathbf{D}(\varphi_s(y_1))) \subseteq \mathbf{V}(J_m^j) \cap \mathbf{D}(y_1)$
	\end{center}
to prove $\psi_s(Z_m^i) = Z_m^j$,
for the isomorphism $\psi_s$ preserve the dimension of a subvariety.
Since $\psi_s(\mathbf{D}(\varphi_s(y_1))) \subseteq \mathbf{D}(y_1)$,
it suffices to show that
	\begin{center}
$\psi_s(\mathbf{V}(J_m^i)) \subseteq \mathbf{V}(J_m^j)$.
	\end{center}
We can easily check that $\varphi_s(L^j) \subseteq L^i$ for triples $(i,j,s)$ as in the statement,
and the assertion follows.
	\end{proof}
	\begin{Coro} \label{varphi act the ideals as image}
We have the following:
	\begin{itemize}
\item[(a')] $\varphi_1(I_m^0) = I_m^0$,
$\varphi_1(I_m^1) = I_m^1$,
$\varphi_1(I_m^2) = I_m^3$ and $\varphi_1(I_m^3) = I_m^2$.
\item[(b')] $\varphi_2(I_m^0) = I_m^0$,
$\varphi_2(I_m^1) = I_m^3$,
$\varphi_2(I_m^2) = I_m^1$ and $\varphi_2(I_m^3) = I_m^2$.
	\end{itemize}
	\end{Coro}
For the proof of our main result in this section,
Theorem \ref{maain result in section4},
we will give a few explicit elements of $I_m^i$ and $I_m^{i,j} := \sqrt{I_m^i + I_m^j}$.

Let
	\begin{center}
$g_1 := -4y_2^2z_2^2 + y_1^2z_3^2 + 4x_3^2z_2 - 4x_2x_3z_3$.
	\end{center}
and
	\begin{alignat}{3}
g_2 :=&\  -y_2^4 - 4y_2^3z_2 + 2y_2^2z_2^2 + 12y_2z_2^3 - 9z_2^4 + 4y_3^2z_1^2 - 8y_3z_1^2z_3 + 4z_1^2z_3^2\\
 &\  + 8x_3^2y_2 - 8x_3^2z_2 - 8x_2x_3y_3 + 8x_2x_3z_3.
	\end{alignat}
We computed the ideals $I_m^1$ and $I_m^2$ for small values of $m$ using Macaulay2,
and found that $g_1$ and $g_2$ belong to $I_m^1$ and $I_m^2$ respectively for $m \geq 5$.
These two elements play important roles in the proof of the main theorem.
For reader's convenience,
we will check this by hand in what follows.

	\begin{Rem} \label{(1)unit in localization (2)relation of I_m^i and J_m^i}
Our strategy for finding elements of $I_m^i$ is as follows.
Let $g \in R_m$ and $n \in \mathbb{Z}_{\geq 0}$.
Suppose $y_1^n g \in J_m^i$.
Then
	\begin{center}
$g \in J_m^i(R_m)_{y_1} \cap R_m = I_m^i$,
	\end{center}
since $y_1$ is a unit in $(R_m)_{y_1}$.
	\end{Rem}
	\begin{Lemma} \label{important elements}
For $m \geq 5$,
we have $g_1 \in I_m^1$ and $g_2 \in I_m^2$.
	\end{Lemma}
	\begin{proof}
First,
we prove $g_1 \in I_m^1$.
Let $\mathbf{a} := x_2t^2 + x_3t^3 + x_4t^4 + x_5t^5$,
$\mathbf{b} := y_1t + y_2t^2 + y_3t^3 + y_4t^4 + y_5t^5$ and $\mathbf{c} := z_2t^2 + z_3t^3 + z_4t^4 + z_5t^5$.
We denote the coefficient of $t^i$ in $f(\mathbf{a}, \mathbf{b}, \mathbf{c})$ by $F^{(i)}$.
Then,
since
	\begin{center}
$\mathbf{x} \equiv \mathbf{a}, \mathbf{y} \equiv \mathbf{b}, \mathbf{z} \equiv \mathbf{c}$ mod $L^1 \cdot R_m[t]/\langle t^6 \rangle$
	\end{center}
(see Notation \ref{about Jet polynomials}),
we have
	\begin{center}
$F^{(i)} \equiv f^{(i)}$ mod $L^1 \cdot R_m[t]/\langle t^6 \rangle$,
	\end{center}
so $F^{(i)} \in  J_m^1$.
We calculate
	\begin{alignat}{4}
F^{(4)} =\ & x_2^2 \ \ && - y_1^2z_2	\\
F^{(5)} =\ & 2x_2x_3 \ \ && -2y_1y_2z_2 - y_1^2z_3,
	\end{alignat}
and then
	\begin{alignat}{6}
F^{(5)2} =\ &\ \ \ \ 4y_1^2y_2^2z_2^2 \ &&+y_1^4z_3^2 - 4x_2x_3y_1^2z_3 \ &&+ 4x_2^2x_3^2 \ &&- 8x_2x_3y_1y_2z_2 + 4y_1^3y_2z_2z_3\\
-4x_3^2F^{(4)} =\ && \ &\ 4x_3^2y_1^2z_2 \ &&- 4x_2^2x_3^2\\
4y_1y_2z_2F^{(5)} =\ &- 8y_1^2y_2^2z_2^2 \ && && &&+ 8x_2x_3y_1y_2z_2 - 4y_1^3y_2z_2z_3
	\end{alignat}
Hence
	\begin{center}
$y_1^2g_1 = F^{(5)2} - 4x_3^2F^{(4)} + 4y_1y_2z_2F^{(5)} \in J_m^1$.
	\end{center}
By Remark \ref{(1)unit in localization (2)relation of I_m^i and J_m^i},
$g_1 \in I_m^1$.

Next we prove $g_2 \in I_m^2$.
We calculate as follows:
By Corollary \ref{varphi act the ideals as image},
we have
	\begin{alignat}{5}
\varphi_2^{-1}(g_1) =&\ \displaystyle -4\left(-\frac{1}{2}y_2-\frac{3}{2}z_2 \right)^2 \left(\frac{1}{2}y_2 - \frac{1}{2}z_2 \right)^2 + \left( -\frac{1}{2}y_1 - \frac{3}{2}z_1 \right)^2 \left( \frac{1}{2}y_3 - \frac{1}{2}z_3 \right)^2\\
 &\ +4x_3^2 \left(\frac{1}{2}y_2 - \frac{1}{2}z_2 \right) - 4x_2x_3 \left(\frac{1}{2}y_3 - \frac{1}{2}z_3 \right) & \ \in I_m^2\\
	\end{alignat}
Since $y_1 - z_1 \in L^2$,
we obtain an element of $I_m^2$ when $y_1$ is replaced by $z_1$.
Then the right hand side equal to
	\begin{alignat}{3}
&\ \displaystyle \frac{1}{4} ( -y_2^4 - 4y_2^3z_2 + 2y_2^2z_2^2 + 12y_2z_2^3 - 9z_2^4 + 4y_3^2z_1^2 - 8y_3z_1^2z_3 + 4z_1^2z_3^2\\
 &\ \  + 8x_3^2y_2 - 8x_3^2z_2 - 8x_2x_3y_3 + 8x_2x_3z_3 )\\
 = &\ \displaystyle \frac{1}{4}g_2.
	\end{alignat}
Thus $g_2 \in I_m^2$.
	\end{proof}
	\begin{Lemma} \label{the corrdinate element of I_m^i,j}
For $m \geq 5$ and $i,j \in \{1,2,3\}$ with $i \neq j$,
we have $y_1, z_1,x_2 \in I_m^{i,j}$,
i.e. $I_m^{i,j} \supseteq L_{322}$.
	\end{Lemma}
	\begin{proof}
Note that,
if $y_1, z_1 \in I$ for an ideal $I$,
then we can easily chech that $y_1, z_1 \in \varphi_s(I)$,
for $s = 1,2$.
We also note that $\varphi_s$ ($s = 1,2$) preserves the elements $x_i$,
for $i = 0, ..., m$.
Hence it is sufficient to show that these three elements belong to $I_m^{1,2}$ by Corollary \ref{varphi act the ideals as image}.
First we check that $y_1,z_1 \in I_m^{1,2}$ i.e.,
$I_m^{1,2} \supseteq L_{222}$.
By the definition of $I_m^i$ and $I_m^{i,j}$,
$I_m^{1,2} = \sqrt{I_m^1 + I_m^2} \supseteq J_m^1 + J_m^2$.
So $y_1$ and $y_1 - z_1$ belong to $I_m^{1,2}$,
i.e. $y_1$ and $z_1$ belong to $I_m^{1,2}$.

Next we show that $x_2$ belongs to $I_m^{1,2}$.
Let $\mathbf{a} = x_2t^2 + x_3t^3 + x_4t^4$,
$\mathbf{b} = y_2t^2 + y_3t^3 + y_4t^4$ and $\mathbf{c} = z_2t^2 + z_3t^3 + z_4t^4$.
Then the coefficient of $t^4$ of $f(\mathbf{a}, \mathbf{b}, \mathbf{c})$ is equal to $x_2^2$.
Moreover,
we have $\mathbf{x} \equiv \mathbf{a}, \mathbf{y} \equiv \mathbf{b}, \mathbf{z} \equiv \mathbf{c}$ mod $L_{222}$,
hence
	\begin{center}
$x_2^2 \equiv f^{(4)}$ mod $L_{222}$,
	\end{center}
and $x_2 \in \sqrt{I_m^{1,2}} = I_m^{1,2}$ i.e.,
$I_m^{1,2} \supseteq L_{322}$.
	\end{proof}
We need the following two lemmas for the proof of the main theorem.
	\begin{Lemma} \label{relation 0 component}
For $m \geq 5$ and for any $i,j \in \{1,2,3\}$ with $i \neq j$,
we have $Z_m^i \cap Z_m^j \subseteq Z_m^0$.
	\end{Lemma}
	\begin{proof}
From Lemma \ref{the corrdinate element of I_m^i,j},
we have $I_m^{i,j} \supseteq L_{322}$ for $i,j \in \{1,2,3\}$ with $i \neq j$.
Hence
	\begin{center}
$I_m^0 = L_{322} + \langle f^{(0)},..., f^{(m)} \rangle = \langle x_0,x_1,x_2,y_0,y_1,z_0,z_1 \rangle + \langle f^{(0)},..., f^{(m)} \rangle \subseteq I_m^{i,j}$.
	\end{center}
	\end{proof}
	\begin{Lemma} \label{maximal intersection}
For $m \geq 5$ and $i, j \in \{1,2,3\}$ with $i \neq j$,
we have $Z_m^0 \cap Z_m^i \supsetneq Z_m^0 \cap Z_m^i \cap Z_m^j$.
	\end{Lemma}
	\begin{proof}
By Proposition \ref{trancefar irreducible components},
we have only to show that $Z_m^0 \cap Z_m^1  \supsetneq Z_m^0 \cap Z_m^1 \cap Z_m^2$.
In other words,
$\sqrt{I_m^0 + I_m^1} \subsetneq \sqrt{I_m^0 + I_m^1 + I_m^2}$.
The proof is divided into two cases,
(a) $m = 5$ and (b) $m \geq 6$.

The case (b).
We prove $y_2 \in \sqrt{I_m^0 + I_m^1 + I_m^2}$ and $y_2 \notin \sqrt{I_m^0 + I_m^1}$.
First we prove $y_2 \in \sqrt{I_m^0 + I_m^1 + I_m^2}$.
From Lemma \ref{important elements},
$g_1 \in I_m^1 \subseteq \sqrt{I_m^0 + I_m^1 + I_m^2}$.
We consider the following two elements modulo $L_{322}$:
	\begin{alignat}{4}
f^{(6)} \equiv&\  - y_2^2z_2 &&\ + z_2^3 + x_3^2\\
g_1 \equiv&\  -4y_2^2z_2^2 &&\ + 4z_2x_3^2.
	\end{alignat}
Then we have
	\begin{center}
$4z_2^4 \equiv 4z_2f^{(6)} - g_1$,
	\end{center}
and the right hand side belongs to $I_m^1$.
Thus $4z_2^4 \in I_m^0 + I_m^1$,
and $z_2 \in \sqrt{I_m^0 + I_m^1}$.
From
	\begin{center}
$f^{(6)} \equiv x_3^2 \ \mathrm{mod}\  \langle L_{322}, z_2 \rangle = L_{323}$,
	\end{center}
we have $x_3 \in \sqrt{I_m^0 + I_m^1}$.
Now from $g_2 \in I_m^2$ and
	\begin{center}
$g_2 \equiv -y_2^4 \ \mathrm{mod}\  \langle L_{323}, x_3 \rangle$,
	\end{center}
it follows that $-y_2^4 \in \sqrt{I_m^0 + I_m^1} + I_m^2$,
hence that $y_2 \in \sqrt{I_m^0 + I_m^1 + I_m^2}$.

Next we prove $y_2 \notin \sqrt{I_m^0 + I_m^1}$.
We consider the point
	\begin{center}
$\mathrm{P'} = (\mathbf{\alpha}, \mathbf{\beta}, \mathbf{\gamma}) = (0, st + t^2, 0)$
	\end{center}
for any $s \in k - \{0\}$.
This point belongs to $\mathbf{V}(J_m^1) \cap \mathbf{D}(y_1)$ from the description of generators of $J_m^1$ in $\CC[x_0,...,x_m,y_0,...,y_m,z_0,...,z_m]$ in Notation \ref{about Jet polynomials} \eqref{OpenPrimeIdealV_m^1}.
When $s$ goes to $0$,
$\mathrm{P'}$ becomes $\mathrm{P} = (0, t^2, 0)$ and this point belongs to $\overline{\mathbf{V}(J_m^1) \cap \mathbf{D}(y_1)} = Z_m^1$.
Moreover $\mathrm{P} \in Z_m^0$,
since $Z_m^0 = \mathbf{V}(I_m^0)$ (Notation \ref{about Jet polynomials} \eqref{PrimeIdealV_m^0}).
Hence $\mathrm{P} \in Z_m^0 \cap Z_m^1$.
Since $y_2 = 1$ at $\mathrm{P}$,
$y_2 \notin \mathbf{I}(Z_m^0 \cap Z_m^1) = \sqrt{I_m^0 + I_m^1}$.

The case (a).
We prove $\mathrm{Q} = (-t^3, -t^2 ,t^2) \in Z_5^0 \cap Z_5^1$,
but $\mathrm{Q} \notin Z_5^0 \cap Z_5^1 \cap Z_5^2$.
First we prove $\mathrm{Q} \in Z_5^0 \cap Z_5^1$.
We can easily check $\mathrm{ord}_{Q}(f) > 5$ and $Q \in \mathbf{V}(L_{322})$,
and we have $\mathrm{Q} \in \mathbf{V}(I_5^0) = Z_5^0$.
To show that $\mathrm{Q} \in Z_5^1$,
we consider $\mathrm{Q'} = (st^2 - t^3, st - t^2, t^2)$ where $s \in k-\{0\}$.
Then $\mathrm{Q'}$ belongs to $\mathbf{V}(J_5^1) \cap \mathbf{D}(y_1)$,
and taking the limit $s \rightarrow 0$,
we get $\mathrm{Q} \in Z_5^1$.
Thus $\mathrm{Q} \in Z_5^0 \cap Z_5^1$.

Next we prove $\mathrm{Q} \notin Z_5^0 \cap Z_5^1 \cap Z_5^2$.
By Lemma \ref{important elements},
$g_2 \in \sqrt{I_5^0 + I_5^1 + I_5^2}$ and $I_5^0 \supseteq L_{322} = \langle x_0, x_1, x_2, y_0, y_1, z_0, z_1 \rangle$,
we have
	\begin{center}
$g_2 \equiv -y_2^4 - 4y_2^3z_2 + 2y_2^2z_2^2 + 12y_2z_2^3 - 9z_2^4 + 8x_3^2y_2 - 8x_3^2z_2\  \mathrm{mod}\  L_{322}$,
	\end{center}
and the right hand side belongs to $\sqrt{I_5^0 + I_5^1 + I_5^2}$.
We set
	\begin{center}
$h := -y_2^4 - 4y_2^3z_2 + 2y_2^2z_2^2 + 12y_2z_2^3 - 9z_2^4 + 8x_3^2y_2 - 8x_3^2z_2$.
	\end{center}
At the point $\mathrm{Q}$,
we have $x_3 = y_2 = -1$ and $z_2 = 1$,
and
	\begin{center}
$h = -(-1)^4 - 4\times(-1)^3\times1 + 2\times(-1)^2\times1^2 + 12\times(-1)\times1^3 - 9\times1^4 + 8\times(-1)^2\times(-1) - 8\times(-1)^2\times 1$\\
$= -32 \neq 0$.
	\end{center}
Hence $\mathrm{Q} \notin Z_5^0 \cap Z_5^1 \cap Z_5^2$.
	\end{proof}
	\begin{Thm} \label{maain result in section4}
Let $X \subset \CC^3$ be the surface defined by $f(x,y,z) = x^2 - y^2z + z^3$,
$X_m^0$ the singular fiber of the $m$-th jet scheme $X_m$ with $m \geq 5$ and $Z_m^0, ... , Z_m^3$ its irreducible components defined in Definition \ref{definition of irreducible components of X_m^0 for D_4-type singular surface}.
Then the maximal elements in $\{ Z_m^i \cap Z_m^j | i \neq j\ (i,j \in \{0,1,2,3\}) \}$ with respect to the inclusion relation are $Z_m^0 \cap Z_m^1$,
$Z_m^0 \cap Z_m^2$ and $Z_m^0 \cap Z_m^3$ and they are pairwise distinct.
	\end{Thm}
	\begin{proof}
By Lemma \ref{relation 0 component} and Lemma \ref{maximal intersection},
for any $i, j \in \{1,2,3\}$ with $i \neq j$,
	\begin{center}
$Z_m^i \cap Z_m^j = Z_m^i \cap Z_m^j \cap Z_m^0 \subsetneq Z_m^i \cap Z_m^0$.
	\end{center}
Hence $Z_m^i \cap Z_m^j$ are not maximal with respect to the inclusion relation.
We show that $Z_m^0 \cap Z_m^i$ is maximal for $i \in \{1,2,3\}$.
If $Z_m^0 \cap Z_m^i \subseteq Z_m^0 \cap Z_m^j$ for $j \in \{1,2,3\}$ and $j \neq i$,
then $Z_m^0 \cap Z_m^i \cap Z_m^j = (Z_m^0 \cap Z_m^i) \cap (Z_m^0 \cap Z_m^j) = Z_m^0 \cap Z_m^i$.
This is a contradiction to Lemma \ref{maximal intersection},
so $Z_m^0 \cap Z_m^i \not\subseteq Z_m^0 \cap Z_m^j$.
Moreover,
if $Z_m^0 \cap Z_m^i \subseteq Z_m^l \cap Z_m^j$ for $j,l \in \{1,2,3\} - \{ i \}$ and $j \neq l$,
then $Z_m^i \subseteq Z_m^0$ since $Z_m^l \cap Z_m^j \subseteq Z_m^0$ by Lemma \ref{relation 0 component}.
This is a contradiction to Proposition \ref{Irreducible decomposition of singular fiber of D_4-type singular point},
so $Z_m^0 \cap Z_m^i \not\subseteq Z_m^l \cap Z_m^j$ for $l \neq i, j$.
Hence $Z_m^0 \cap Z_m^i$ are maximal with respect to inclusion relation and are pairwise distinct for $i = 1,2,3$.
	\end{proof}
	\begin{Coro}
The graph obtained by Construction \ref{ConstructionofGraph} for $m \geq 5$ is the resolution graph of a $D_4$-type singularity.
	\end{Coro}
	\begin{proof}
The set $E$ is $\{Z_m^0 \cap Z_m^1, Z_m^0 \cap Z_m^2, Z_m^0 \cap Z_m^3 \}$,
so the graph obtained by Construction \ref{ConstructionofGraph} is as follows:
	\[
		\xymatrix{
			 &  & Z_m^2\\
			Z_m^1 \ar@{-}[r] & Z_m^0 \ar@{-}[ru] \ar@{-}[rd] & \\
			 & & Z_m^3.
		}
	\]
	\end{proof}

\end{document}